\documentclass[reqno]{amsart}

\input xy
\xyoption{all}
\usepackage{accents}
\usepackage{epsfig}
\usepackage{color}
\usepackage{amsthm}
\usepackage{amssymb}
\usepackage{amsmath}
\usepackage{amscd}
\usepackage{amsopn}
\usepackage{graphicx}
\usepackage{xspace}

\usepackage{url}
\usepackage{enumitem, hyperref}\hypersetup{colorlinks}

\usepackage{color} 

\definecolor{darkred}{rgb}{1,0,0} 
\definecolor{darkgreen}{rgb}{0,0.8,0}
\definecolor{darkblue}{rgb}{0,0,1}

\hypersetup{colorlinks,
linkcolor=darkblue,
filecolor=darkgreen,
urlcolor=darkred,
citecolor=darkgreen}

\makeatletter
\def\reflb#1#2{\begingroup
    #2%
    \def\@currentlabel{#2}%
    \phantomsection\label{#1}\endgroup
}
\makeatother
\theoremstyle{definition}
\newtheorem{defn}{Definition}

\theoremstyle{plain}
\newtheorem{theorem}{Theorem}[section]
\newtheorem{lemma}[theorem]{Lemma}

\newtheorem{corollary}[theorem]{Corollary}
\theoremstyle{remark}
\newtheorem{remark}[theorem]{Remark}
\title{On the dynamics characterization of complex projective spaces}
\author{Mita Banik}
\date{\today}
\begin{document}
\address{ Department of Mathematics, UC Santa Cruz, Santa
  Cruz, CA 95064, USA} \email{mbanik@ucsc.edu}

\begin{abstract}
    We show that a closed weakly-monotone symplectic manifold of dimension $2n$ which has minimal Chern number greater than or equal to $n+1$ and admits a Hamiltonian toric pseudo-rotation is necessarily monotone and its quantum homology is isomorphic to that of the complex projective space. As a consequence when $n=2$, the manifold is symplectomorphic to $\mathbb{C}P^2$.
\end{abstract}
\maketitle

\section{Introduction}
\smallskip
We show that a closed weakly-monotone symplectic manifold $M^{2n}$, which has minimal Chern number $N \geq n+1$ and admits a Hamiltonian toric pseudo-rotation is necessarily monotone and its quantum homology is isomorphic to the quantum homology of $\mathbb{C}P^n$.

There are several definitions of pseudo-rotations, but roughly speaking pseudo-rotations are Hamiltonian diffeomorphisms with finite and minimal possible number of periodic points. They have been studied in \cite{GG1} and the references therein.

By definition, a pseudo-rotation $\varphi$ is toric if at one of its fixed points the eigenvalue of $D\varphi$ satisfy no resonance relations beyond the conditions that they come in complex conjugation pairs. To be more premise, the requirement is that the semi-simple part of $D\varphi$ topologically generates an $n$-dimensional torus in $Sp(2n)$. For instance, pseudo-rotations obtained by the conjugation method from toric symplectic manifolds are toric. While the toric condition appear generic, in fact the very existence of a toric pseudo-rotation $\varphi$ imposes strong restrictions on the symplectic topology of the manifold $M$. (For example, when $\varphi$ is a toric true rotation, essentially by definition  $M$ is toric). Connections between pseudo-rotations and holomorphic curves in $M$ have been recently studied in \cite{CGG}, \cite{CGG1}, \cite{S1} and \cite{S2}. Our goal in this paper is to further explore the relations between toric pseudo-rotation and the symplectic topology of $M$.

 In \cite{CGG}, it has been have shown that the quantum product is deformed (i.e. not equal to the intersection product) when a weakly-monotone symplectic manifold admits a Hamiltonian toric pseudo-rotation. We use the tools of extremal partitions introduced there to find constraints on the minimal Chern number $N$ for such a manifold and further prove that when $N \geq n+1$, the quantum homology is isomorphic to that of $\mathbb{C}P^n$. As a corollary of our result and \cite{OO1}, \cite{OO2} when $n=2$, $M$ is symplectomorphic to $\mathbb{C}P^2$.

 The organization of the paper is as follows. In section 2, we state the definitions and the main results. The proofs appear in section 4. In section 3 we briefly recall the conventions and preliminaries on symplectic topology and extremal partitions.\\
 
 \textbf{Acknowledgments}  The author would like to thank Viktor Ginzburg for his helpful advice and guidance throughout the writing of this paper and Felix Schlenk for useful remarks.
\medskip
\section{Main results}
\medskip 
The main goal of this paper is to understand the quantum homology of a weakly-monotone symplectic manifold, which admits a toric pseudo-rotation.
\begin{defn}(Pseudo-rotation) A Hamiltonian diffeomorphism $\varphi: M \to M$ is a pseudo-rotation (over the base field $\mathbb{F}$) if $\varphi$ is strongly non-degenerate i.e. all iterates $\varphi^k$ are non-degenerate, and the differential in the Floer complex of $\varphi^k$ over $\mathbb{F}$ vanishes for all $k \in \mathbb{N}$.
    
\end{defn}
\begin{defn} (Toric Pseudo-rotations) A pseudo-rotation $\varphi$ of a closed symplectic manifold $M^{2n}$ is said to be $toric$ if it has a elliptic fixed point $x$ with $\dim \Gamma(x) = n$.
\end{defn}
Here $\Gamma(x)$ is a compact abelian subgroup of $Sp(2n)$ generated by $\Tilde{P}$, where $\Tilde{P}$ is isospectral to $P= D \varphi |_x$ and semisimple. Alternatively, since $P$ is elliptic all the eigenvalues of $P$ lie on the unit circle. Let $\Vec{\theta} := ( \theta_1,...,\theta_n)$ $\in \mathbb{T}^n$ be the collection of Krein-positive eigenvalues of $P$ , ref. \cite{SZ}. The group $\Gamma(x)$ is naturally isomorphic to the subgroup
of the torus $\mathbb{T}^n$ generated by $\Vec{\theta}$. Then the above definition of being a toric pseudo-rotation is equivalent to that the sequence $\{ k\Vec{\theta} $ $|$ $ k \in \mathbb{N} \}$ $\subset \mathbb{T}^n$ is dense in $\mathbb{T}^n$.

Let us now state the main results. The conventions and preliminaries have been reviewed in Section 3. We fix base field $\mathbb{F} =\mathbb{Z}_2$ for our coefficients.
\begin{theorem}\label{2}
Assume that a weakly-monotone symplectic manifold $M^{2n}$ admits a toric pseudo-rotation with minimal Chern number $N \geq n+1$. Then $N = n+1$, $M$ is monotone and the quantum homology $HQ_*(M)$ is isomorphic to $HQ_*(\mathbb{C}P^n)$.
\end{theorem}
\begin{remark}
Note that $N \geq n+1$ in all known examples of closed monotone manifolds and $\mathbb{C}P^n$ is the only such known manifold with $N=n+1$. (However proving this in the symplectic setting for $n>2$ appears to be currently out of reach). However, in Thm \ref{2} the manifold $M$ is not a priori assumed to be monotone and thus can have large $N$. For monotone manifolds admitting a pseudo-rotation, $N \leq 2n$ (\cite{CGG}, \cite{S1}), and $N \leq n+1$ in all known examples of weakly monotone manifolds with pseudo-rotations. Thus Thm \ref{2} establishes in particular the latter fact for toric pseudo-rotations.
\end{remark}

In \cite{OO1}, \cite{OO2} Ohta-Ono proved proved that the diffeomorphism type of any closed monotone
symplectic 4-manifold is $\mathbb{C}P^1 \times \mathbb{C}P^1$ and $\mathbb{C}P^2 \# k  \overline{\mathbb{C}P^2}$ for $0 \leq k \leq 8$ based on the work of McDuff \cite{McD} and Taubes \cite{Tau}. We also have the uniqueness of monotone symplectic structures on $\mathbb{C}P^1 \times \mathbb{C}P^1$ and $\mathbb{C}P^2 \# k \overline{\mathbb{C}P^2}$ for $0 \leq k \leq 8$ (see the survey \cite{Sa}). The quantum homology $HQ_*(M)$ is isomorphic to $HQ_*(\mathbb{C}P^n)$. Thus, as a consequence of our result and the references presented above, we have the following corollary.

\begin{corollary} Assume that a closed connected symplectic 4-manifold with minimal Chern number $N \geq 3$ admits a toric pseudo-rotation. Then $M$ is symplectomorphic to $\mathbb{C}P^2$. \qed
\end{corollary}
\smallskip
\section{Preliminaries}
\medskip 
Throughout the paper we will assume $(M^{2n},\omega)$ to be a weakly monotone closed symplectic manifold , ref. \cite{HS}. The minimal Chern number, i.e. the positive generator of the group $\langle c_1(TM), \pi_2(M) \rangle \subset \mathbb{Z}$ is denoted by $N$.

A Hamiltonian diffeomorphism is the time-one map $\varphi = \varphi_H$ of the time-dependent flow $\varphi^t_H$ of a 1-periodic in time Hamiltonian $H : S^1 \times M \to \mathbb{R}$.

We briefly discuss the notations and conventions for Floer homology and quantum homology, see \cite{HS} and \cite{EP} for more details. 

Let $H_*(M) := H_*(M,\mathbb{Z}_2)$ and $H^S_2(M,\mathbb{Z})$ be the group of integral spherical homology classes, i.e. the image of the Hurewicz homomorphism $\pi_2(M) \to H_2(M,\mathbb{Z})$.
Set \begin{center}
     $\bar{\pi}_2(M) = H^S_2(M,\mathbb{Z})/ \sim$,
\end{center}
where by definition $A \sim B$ iff $\omega(A) = \omega(B)$ and $c_1(A) = c_1(B)$. Here $\omega(-)$ and $c_1(-)$ are the integrals of $\omega$ and $c_1(TM)$ over the spherical homology classes.

Denote $\Gamma = [\omega](H^S_2(M,\mathbb{Z})) \subset \mathbb{R}$ the subgroup of periods of the
symplectic form on $M$ on spherical homology classes. Let $s$ and $ q$ be formal variables.  Define the field $K_\Gamma$ whose elements are generalized Laurent series in $s$ of the following form:\begin{align}
    K_\Gamma = \{ \sum_{\theta \in \Gamma} z_\theta s^{\theta}, \,  z_\theta \in \mathbb{Z}_2 , \, \#\{\theta > c \mid z_\theta \neq 0 \} < \infty , \, \forall c \in \mathbb{R}  \,  \}.
\end{align}

Define a graded ring  $\Lambda_\Gamma := K_\Gamma[q,q^{-1}]$ by setting the degree of $s$ to be zero and the degree of $q$ to be $2N$. (Note that the latter convention differs from that in \cite{EP} where the deg of $q$ is $2$.)

The (small)\textit{quantum homology} of $M$ is denoted by $HQ_{*}(M)$. This is a graded algebra over the Novikov ring $\Lambda_\Gamma$ (\cite{EP}) and as a $\Lambda_\Gamma$-module $HQ_{*}(M) = H_{*}(M) \otimes_{\mathbb{Z}_2} \Lambda_\Gamma$. The grading on $HQ_*(M)$ is given by the gradings on $H_*(M)$ and $\Lambda_\Gamma$ :\begin{center}
    $\deg(a  \otimes z_\theta s^\theta q^m) = \deg(a) + 2Nm$.
\end{center}

The algebra $HQ_{*}(M)$ is equipped with quantum product: given $a \in H_k(M)$ and $b \in H_l(M)$, their quantum product is a class $a * b \in HQ_{k+l -2n}(M)$ such that \begin{align*}
   a * b = \sum_{A\in \bar{\pi_2}(M)} (a *b)_A \otimes s^{-\omega(A)} q^{-c_1(A)/N}
\end{align*} where $(a * b)_A \in H_{k+l-2n+2c_1(A)}(M)$ is defined by \begin{align*}
    (a * b)_A \circ c= GW^{\mathbb{Z}_2}_A(a,b,c) \; , \forall c \in H_{*}(M).
\end{align*} Here $\circ$ is the intersection product and $GW^{\mathbb{Z}_2}_A(a,b,c)$ denotes the Gromov-Witten invariant.

 The Floer complex and the Floer homology are denoted by  $CF_{*}(\varphi,\Lambda_\Gamma)$ and $HF_{*}(\varphi,\Lambda_\Gamma)$ respectively. Let $\Tilde{P}(H)$ be the free $\mathbb{Z}_2$-module generated by the set of capped one-periodic orbits $\Tilde{x}$ of $H$. Consider the free $\Lambda_\Gamma$-module $\Tilde{P}(H) \otimes_{\mathbb{Z}_2} \Lambda_\Gamma$ and let $R$ be a $\Lambda_\Gamma$-sub module of $\Tilde{P}(H) \otimes_{\mathbb{Z}_2} \Lambda_\Gamma$ generated by $A\#\Tilde{x} \otimes 1- \Tilde{x} \otimes s^{\omega(A)}q^{c_1(A)/N}$, $A \in \bar{\pi}_2(M)$. 
 
 The grading on $\Lambda_\Gamma$ and the grading $\mu$ on $\Tilde{P}(H)$ given by the Conley-Zehnder index give rise to the grading \begin{center}
     $\deg(\Tilde{x} \otimes z_\theta s^\theta q^m) = \mu(\Tilde{x}) + 2Nm$.
 \end{center} 
 Then $\deg(A\#\Tilde{x} \otimes 1) = \deg(\Tilde{x} \otimes s^{\omega(A)}q^{c_1(A)/N}) = \mu(A \#\Tilde{x})$. Hence we get the graded $\Lambda_\Gamma$-module $CF_{*}(\varphi,\Lambda_\Gamma) \; := \Tilde{P}(H) \otimes_{\mathbb{Z}_2} \Lambda_\Gamma$/$R$ and the Floer homology $HF_{*}(\varphi,\Lambda_\Gamma)$ is defined as usual.\\

For a pseudo-rotation we have the following isomorphisms where the first one is natural and the second one is due to \cite{PSS},\begin{center}
   $ CF_{*}(\varphi,\Lambda_\Gamma) \cong HF_{*}(\varphi,\Lambda_\Gamma) \cong HQ_{*}(M)[-n] $.
\end{center} 

\subsection{Extremal Partitions}
Fix a path $\Phi \in \widetilde{Sp}(2n)$ and assume $\Phi$ is elliptic and strongly non-degenerate. We briefly recall extremal partition here, for more details see \cite{CGG}.
\begin{defn}
    A partition $k_1+ k_2 +\dots +k_r = k$ , $k_i \in \mathbb{N}$, of length $r$ is said to be \textit{extremal} (with respect to $\Phi$) if \begin{center}
    $\mu(\Phi^{k_1}) + \dots + \mu(\Phi^{k_r}) - \mu(\Phi^{k}) = (r-1)n$, \\
    \end{center}
    where $\mu$ is the Conley-Zehnder index of the path.
 
\end{defn}
 Let $\Tilde{x}$ be a capped one-periodic orbit of $\varphi$; denote $\Tilde{x}_{k_i} = \Tilde{x}^{k_i}$ which is  a capped periodic orbit for $\varphi^{k_i}$. The only zero-energy pair-of-pants curve from $(\Tilde{x}_{k_1}, \dots , \Tilde{x}_{k_r})$ to $\Tilde{x}_k$ where  $k = k_1+ k_2 +\dots +k_r $, is the constant curve. The moduli space of such curves have virtual dimension zero and the constant curve from $(\Tilde{x}_{k_1}, \dots , \Tilde{x}_{k_r})$ to $\Tilde{x}_k$ is regular when the partition $k = k_1+ k_2 +\dots +k_r $ is extremal. Therefore, \begin{center}
    $\Tilde{x}_{k_1}* \dots* \Tilde{x}_{k_r} = \Tilde{x}_{k} + \dots \neq 0$,
\end{center} where $*$ denotes the pair-of-pants product in global Floer homology.

The above arguments have been summarized by \cite{CGG}, in the following theorem which we recall.
\begin{theorem}\label{1} Let $\Tilde{x}$ be a capped one-periodic orbit of a pseudo-rotation $\varphi$, and let $k = k_1+ k_2 +\dots +k_r$ be an extremal partition of length $r$ with respect to $\Phi =D\varphi^t|_{\Tilde{x}}$. Set $\alpha_i = [\Tilde{x}^k_i \otimes 1] \in HQ_{*}(M)$. Then $|\alpha_i| = n+\mu(\Phi^{k_i})$ and the following holds \begin{center}
    $\alpha_1 * \dots * \alpha_r \neq 0 .$
\end{center} 
\end{theorem} 
\bigskip
\section{Proofs}
\smallskip
The first step in proving the theorem is showing the invertibility of the $[pt] \otimes 1$ class with respect to the quantum product in $HQ_*(M)$. We begin by establishing the following lemma.
\begin{lemma} Let $\varphi$ be  a toric pseudo-rotation of $M^{2n}$  with minimal Chern number $N \geq n+1 $. Then $([pt] \otimes 1)^r \neq 0$ for all $r \geq 1$.
\end{lemma}
\begin{proof}
Let $\Phi$ be the linearized flow $D\varphi^t|_{\Tilde{x}}$ along a capped orbit $\Tilde{x}$ such that $\dim\Gamma(x)=n$, i.e. $\Gamma(x) = \mathbb{T}^n$.

When $\Phi(1)$ is semi-simple we can decompose $\Phi$ as a product $\phi \xi$ where $\phi$ is a loop and $\xi$ is a direct sum of $n$ ``short paths" $t \mapsto exp(\pi \sqrt{-1} \lambda t)$ where $t \in [0,1)$ and $|\lambda| < 1$, ref, [\cite{GG}, Sect. 4]. We set $loop(\Phi) := \hat{\mu}(\phi)$ where $\hat{\mu}(\phi)$ is the mean index, which is twice the Maslov index of $\phi$. For any iteration $\Phi^k(1)$ we have $\mu(\Phi^k) = k \, loop(\Phi) + \mu(\xi^k).$

All the eigenvalues of $\Phi(1)$  are necessarily distinct and in particular $\Phi(1)$ is semi-simple. Therefore, $\Phi(1)^k$ $k \in \mathbb{N}$ is dense in $\mathbb{T}^n$. For some $l \in \mathbb{N}$, we can choose $\Phi(1)^l$ to be sum of small rotations $exp(\pi \sqrt{-1} \theta_i)$ such that $\theta_i <0$ for all $i =1,...,n$ and $\theta_i$'s are very close to each other. Iterating again to bring $exp(\pi \sqrt{-1} m\theta_i)$ close to $1 \in S^1$. The choice of $``m$'' is such that the $loop(\Phi^m) = -2n + d $ (with $ 2N | d$) and if $\lambda_i$'s be the end points of the ``short paths" in  $\xi^m$ then we have $\lambda_i >0$ and small. We have 
\begin{center}
    $\mu(\Phi^m) = -2n + d + n = -n + d$.
\end{center}
Also ensuring $r$max$|\lambda_i|<2$ we have, \begin{center}
    $\mu(\Phi^{rm}) = r(-2n+d) + n$.

\end{center}
And $m+...+m = rm$ is an extremal partition since \begin{center}
    $r\mu(\Phi^m) - (r-1)n = r(-n +d) - (r-1)n= r(-2n +d ) +n =  \mu(\Phi^{rm})$.
\end{center}
Since $\mu(\Phi^m)= \mu(\Tilde{x}^m) = -n $ (mod $ 2N)$  therefore for some element $\lambda \in K_\Gamma$, $ \Tilde{x}^m \otimes \lambda = [pt] \otimes 1$  (since $N \geq n+1$). Thus $([pt] \otimes 1)^r \neq 0$ by Theorem \ref{1}.
\end{proof}
\begin{corollary}
Assume that a weakly-monotone symplectic manifold $M^{2n}$ admits a toric pseudo-rotation and $N \geq n+1$, then $[pt \otimes 1]$ is invertible and the $[pt \otimes 1]$ class satisfies the following conditions:\begin{align}
    ([pt] \otimes 1)^N =  [M] \otimes \alpha ,                                                      
\end{align} where $\alpha$ is invertible in $K_\Gamma$ and $deg(\alpha) = -2Nn$.
\end{corollary}
\begin{proof}
By the previous lemma we have $([pt] \otimes 1)^N \neq 0$. We have deg($[pt] \otimes 1)^N$ =$-2n(N-1)$ and since $N \geq n+1$, therefore $([pt] \otimes 1)^N =  [M] \otimes \alpha $ where $deg(\alpha) = -2Nn$. The class $\alpha$ is of the form $(\sum_{\theta \in \Gamma} z_\theta s^{\theta})q^{-n}$. Therefore $\alpha$ is invertible since $K_{\Gamma}$ is a field.
\end{proof}
The invertibility of the $[pt] \otimes 1$ class is crucial to the arguments below. By doing a simple degree analysis we immediately get some obstructions to the minimal Chern number. For the later part of the theorem invertibility gives us uniqueness of the homology classes. \\ \\
\medskip
\textbf{Proof of Theorem \ref{2}} 
\begin{proof}
Let us begin by proving that when a weakly-monotone $M^{2n}$ admits a toric pseudo-rotation, then $N \leq n+1$. We recall from \cite{CGG} that when $M$ admits a toric pseudo-rotation there is a non zero class $ u \in H_{2n-2}(M)$ such that $([u] \otimes 1)^r \neq 0$ for every $r \geq 1$. Since $[pt]\otimes 1$ satisfies (2), $([pt] \otimes 1)*([u] \otimes 1) \neq 0$. By doing degree computation we see $deg(([pt] \otimes 1)*([u] \otimes 1)) = -2$ thus if $N> n+1$, then $([pt] \otimes 1)*([u] \otimes 1)= 0$ which is a contradiction.\\

For the final part of the proof we will show that when $N = n+1$, then \begin{center}
    $ \dim(H_{2n-2i}(M)) = 1$ for $1 \leq i \leq n$.
\end{center}

Let us first establish the result for  $i=n-1$. \\

Let $u_1, u_2$ be two non-zero classes in $H_{2}(M)$, consider $([pt] \otimes 1)^{n-1}*([u_i] \otimes 1)$. Now $deg(([pt] \otimes 1)^{n-1}*([u_i] \otimes 1))= 2n + -2(n+1)(n-1)$ and $([pt] \otimes 1)^{n-1}*([u_i] \otimes 1) \neq 0$ since $[pt]\otimes 1$ satisfies (2). Therefore $([pt] \otimes 1)^{n-1}*([u_i] \otimes 1) =  [M] \otimes \lambda_i$ for some invertible element $\lambda_i \in K_\Gamma$.

Multiplying $[pt]^2 \otimes 1$  with both sides and using invertibility of $\alpha$
we obtain $[u_1] \otimes 1 = [u_2] \otimes \lambda_2^{-1}\lambda_1 $. This shows the classes $[u_1]$ and $[u_2]$ are linearly dependent, hence the dimension of $H_{2}(M)$ is 1. This also implies that $\dim H_2(M,\mathbb{Z}) =1$ and thus $M$ is monotone.\\

Let class $A_0$ be the obvious generator for $H^S_2(M,\mathbb{Z})$. We set $q' = (s^{-\omega(A_0)}q^{-1})$ with $deg(q') = -2(n+1)$ We will rename $q'$ by $q$ which is the generator of the Novikov ring and denote $A\otimes \alpha$ by $\alpha A$ for $\alpha \in K_\Gamma$ and $A \in H_{*}(M)$.\\

Now let us prove the result for $i > 1$. We have $u^i \neq 0 $ with $deg(u^i) = 2n -2i$. Let $\beta$ be another non-zero class in  $H_{2n-2i}(M)$, then $deg([pt]^i * \beta) = 2n -2i -2ni = 2n -2i(n+1) $ and thus $[pt]^i * \beta = \lambda q^i M$. By similar arguments as above and using (2), $\beta$ and $u^i$ are linearly dependent and hence the dimension of  $H_{2n-2i}(M)$ is $1$. \\

 So by above arguments it follows that  $HQ_*(M$) is generated by $u \in H_{2n-2}(M)$. The identity $u^{n+1} = q[M]$ readily follows since  $deg(u^{n+1})= -2 = -2(n+1) + 2n$, and the theory of extremal partition asserts the coefficient is $1$. This establishes the isomorphism with that of the quantum homology of $\mathbb{C}P^n$.
\end{proof}

\end{document}